\documentclass[12pt,reqno]{amsart}
\usepackage{color}
\newcommand{\bydef}{\stackrel{\rm def}{=}}

\def ~{\hspace{1mm}}

\newcommand{\clm}{{\mathcal{M}}}

\newcommand{\Dbar}{\overline{\mathbb D}}

\newcommand\tr{{\mbox{tr}}}
\newtheorem{thm}{Theorem}[section]

\newtheorem{lem}[thm]{Lemma}

\newtheorem{defn}[thm]{Definition}

\numberwithin{equation}{section}

\def\textmatrix#1&#2\\#3&#4\\{\bigl({#1 \atop #3}\ {#2 \atop #4}\bigr)}
\def\dispmatrix#1&#2\\#3&#4\\{\left({#1 \atop #3}\ {#2 \atop #4}\right)}

\begin{document}
\title[Symmetrized bidisk]
{holomorphic functions on the symmetrized bidisk - realization, interpolation and Extension}

\author[Bhattacharyya]{Tirthankar Bhattacharyya}
\address[Bhattacharyya]{Department of Mathematics, Indian Institute of Science, Banaglore 560 012}
\email{tirtha@member.ams.org}

\author[Sau]{Haripada Sau}
\address[Sau]{Department of Mathematics, Indian Institute of Science, Banaglore 560 012}
\email{sau10@math.iisc.ernet.in}

\thanks{This research is supported by University Grants Commission, India via CAS.}

\maketitle
\begin{abstract}{There are three new things in this paper about the open
            symmetrized bidisk $\mathbb G = \{ (z_1+z_2, z_1z_2) : |z_1|, |z_2| < 1\}$. They are motivated in the Introduction. In this Abstract, we mention them in the order in which they will be proved.

\begin{enumerate}

\item The Realization Theorem: A realization formula is demonstrated for every $f$ in the norm unit ball of $H^\infty(\mathbb G)$.
\item The Interpolation Theorem: Nevanlinna-Pick interpolation theorem is proved for data from the symmetrized bidisk and a specific formula is obtained for the interpolating function.
    \item The Extension Theorem: A characterization is obtained of those subsets $V$ of the open
            symmetrized bidisk $\mathbb G$ that have the
            property that every function $f$ holomorphic in
            a neighbourhood of $V$ and bounded on $V$ has an $H^\infty$-norm preserving extension to
the            whole of $\mathbb G$. \end{enumerate}

    }

\end{abstract}

\section{Introduction}

\subsection{Extension} Theorem B of H. Cartan about sheaf cohomology
on Stein domains implies the following as a special case.

\vspace*{5mm}

\noindent \textbf{Theorem (H. Cartan).} {\em
If $V$ is an analytic variety in a domain of holomorphy $\Omega$ and if $f$ is a holomorphic function on $V$, then there is a holomorphic function $g$ on $\Omega$ such that $g=f$ on $V$.}

\vspace*{5mm}

About extensions which are
not just holomorphic, but also norm preserving, the notable
success has been extension of holomorphic functions from
submanifolds of Stein manifolds with weighted $L^2$ estimates
\cite{OT}. Attempts of extension of bounded holomorphic functions
from submanifolds, preserving $H^\infty$-norm have required
stringent sufficient conditions, see \cite{Alex} and
\cite{Henkin}.

The main result of this note relates this question of
extension to Hilbert space operator theory. When $\Omega$ is the symmetrized bidisk
$$ \mathbb G = \{ (z_1+z_2, z_1z_2) : |z_1|, |z_2| < 1\}$$
and $V$ is a subset of $\Omega$, we find a property of $V$ that is necessary and sufficient to ensure that
every function holomorphic in a neighbourhood of $V$ and bounded on $V$ extends to the whole of the summetrized bidisk in such a way
that the $H^\infty$-norm of the original function on $V$ is not increased. The symbol Hol$^\infty(V)$
stands for those bounded functions $f$ on $V$ which have a holomorphic extension to a neighbourhood of $V$.

Let $\mathcal A$ be a subset of Hol$^\infty(V)$.
We shall explain two properties of the set $V$ below - the $\mathcal A$-extension property and the property of being an $\mathcal A$-von Neumann set.
The $\mathcal A$-extension property means that whenever $f \in
\mathcal A$, there is a bounded holomorphic function $g$ on whole
of $\mathbb G$ such that \begin{equation} \label{Aext}  g|_V = f
\mbox{ and } \sup_{\mathbb G} |g| = \sup_V |f|. \end{equation}
Note that an extension of the form (\ref{Aext}) is what we want to achieve, motivated by the theorem of Cartan. The challenge is to decide what kind of sets $V$ will allow us that.

The motivation for defining an $\mathcal A$-von Neumann set comes from the 1951 paper of von Neumann where he showed that for a contraction $T$ on a Hilbert space and a polynomial $p$, the following inequality is satisfied.
$$ \| p(T) \| \le \sup_{z \in \mathbb D} |p(z)|.$$
A dozen years later, Ando came up with an elegant generalization of this inequality. If $(T_1, T_2)$ is a commuting pair of contractions, and $p$ is a polynomial in two variables, then
$$ \| p(T_1, T_2) \| \le \sup_{z_1, z_2 \in \mathbb D} |p(z_1, z_2)|.$$
A polynomially  convex compact set $X \subseteq \mathbb C^2$ is called a spectral set for a pair $(T_1, T_2)$ of
commuting bounded operators if $\sigma (T_1, T_2) \subseteq X$ and
$$ \| p(T_1, T_2) \| \le \sup_X | p |$$
for any polynomial $p$ in two variables. Put in this way, a pair
of commuting contractions is, by Ando's inequality, the same as a
commuting pair of bounded operators which has the closed bidisk as a spectral set. The
symmetrized bidisk is a non-convex, but polynomially convex subset of
$\mathbb C^2$. Its geometry has been studied in \cite{ay-blm} and
\cite{ay-jga}. Study of commuting operator pairs which have the symmetrized
bidisk as a spectral set has been extensively carried out in
\cite{ay-jfa}, \cite{ay-ems}, \cite{ay-jot} and \cite{BPSR}.

\begin{defn} \label{Gamma-contraction}
A pair of commuting bounded operators $(S, P)$ on a Hilbert
space $\mathcal H$ having the closed symmetrized bidisk $\Gamma$
as a spectral set is called a $\Gamma$-contraction. Thus $(S,P)$
is a $\Gamma$-contraction if and only if $\| f(S,P) \| \le
\sup_{\mathbb G} |f|$ for all polynomials $f$ in two variables.
\end{defn}

This terminology is due to Agler and Young. Many examples of
$\Gamma$-contractions are discussed in \cite{BPSR}.

It is advantageous to broaden the class of functions to include holomorphic functions, especially since a functional calculus is available. If $V \subseteq \mathbb C^2$, say that a pair of commuting operators $(T_1, T_2)$ on a Hilbert space is $subordinate$ to $V$ if the Taylor joint spectrum $\sigma (T_1, T_2) \subseteq V$ and $g(T_1, T_2) = 0$ whenever $g$ is holomorphic in a neighbourhood of $V$ and $g|_V = 0$. If $f$ is a function on $V$ that has a holomorphic extension in a neighbourhood of $V$ and $(T_1, T_2)$ is subordinate to $V$, define $f(T_1, T_2)$ by setting $f(T_1, T_2) = g(T_1, T_2)$ where $g$ is any holomorphic extension of $f$ in a neighbourhood of $V$. Given $\mathcal A$ as above, $V$ is called an $\mathcal A$-von Neumann set if for
any $\Gamma$-contraction $(S,P)$ subordinate to $V$ and any $f \in \mathcal A$,
$$ \| f(S, P) \| \le \sup_V |f|.$$

The main discovery of this paper is the following theorem.

\vspace*{5mm}

\noindent \textbf{Extension Theorem}. {\em
Let $V \subseteq \mathbb G$. Let $\mathcal A \subseteq$
Hol$^\infty(V)$. Then $V$ has the $\mathcal A$-extension property
if and only if $V$ is an $\mathcal A$-von Neumann set.}

\vspace*{5mm}

Note that $V$ is not assumed to be a variety in this. An analogous theorem for the bidisk was proved by Agler and McCarthy in \cite{AMc-Ann}.

\subsection{Interpolation} The route to proving the Extension Theorem
is through an Interpolation Theorem which has its own independent interest.
Interpolation means the following. Given $n$ points $\lambda_1,
\lambda_2, \ldots , \lambda_n$ in $\Omega$ and $n$ points $w_1,
w_2, \ldots , w_n$ in the closed unit disk in the plane, we want
to know whether there is an $H^\infty(\Omega)$ function $f$ such
that $f(\lambda_i) = w_i$ for $i=1,2, \ldots , n$. Pick's
classical theorem tells us that when $\Omega = \mathbb D$, this
happens if and only if the $n \times n$ matrix
$$ \left( \left( \; \frac{1 - w_i \bar{w}_j}{1 - \lambda_i \bar{\lambda}_j} \; \right) \right) = \left( \left( \; (1 - w_i \bar{w}_j) k( \lambda_i , \lambda_j)\; \right) \right), $$
where $k$ is the Szego kernel or the reproducing kernel of the Hardy space of the disk, is non-negative definite. The matter is not so straightforward in, say, the bidisk. Although there is a well-studied Hardy space whose reproducing kernel is
$$\frac{1}{(1 - z_1\bar{w}_1)(1 - z_2\bar{w}_2)}$$
no formulation of a criterion for interpolation in terms of this kernel alone is known. Instead, Agler proved a criterion in terms of a family of kernels.

\begin{defn} A scalar valued function $k$ on $\Omega \times \Omega$ is called a kernel (respectively a weak kernel) if $\sum_{i,j=1}^n c_i \bar{c}_j k(z_i, z_j) > 0$ (respectively $\sum_{i,j=1}^n c_i \bar{c}_j k(z_i, z_j) \ge 0$ ) for any positive integer $n$, any $n$ points $z_1, z_2, \ldots ,z_n$ in $\Omega$ and any $n$ scalars $c_1, c_2, \ldots ,c_n$, not all zeros. If, moreover, $k$ is holomorphic in the first variable and anti-holomorphic in the second variable, then it is called a holomorphic kernel (respectively a holomorphic weak kernel). \label{kernel} \end{defn}

Given a kernel $k$, there is a Hilbert space of functions $H_k$ such that the family of functions $\{ k(\cdot , w): w \in \Omega\}$ is contained in $H_k$, is a total set in $H_k$ and has the reproducing property, i.e.,
$$f(z) = \langle f , k(\cdot , z)\rangle$$
for an $f$ in $H_k$ and any $z$ in $\Omega$. Because of this reproducing property, the Hilbert space $H_k$ is called a reproducing kernel Hilbert space. If $k$ is a holomorphic kernel, then $H_k$ consists of holomorphic functions.

Agler showed that there is an $f$ in $H^\infty(\mathbb D^2)$ with $f(\lambda_i) = w_i$ for $i=1,2, \ldots , n$ if and only if
$$ \left( \left( \; (1 - w_i \bar{w}_j) k( \lambda_i , \lambda_j)\; \right) \right)_{i,j=1}^n $$
is a weak kernel for every kernel $k$ on the bidisk which has the property that the co-ordinate multiplications are contractions on $H_k$. This showed the need for more than one kernel.

A multiplier on the reproducing kernel Hilbert space $H_k$ is a holomorphic function $\varphi$ defined on $\mathbb G$ such that the multiplication operator
 $$ M_\varphi : f \rightarrow \varphi f$$
 is a bounded operator on $H_k$. Of particular importance to us will be the following multipliers.
\begin{equation} (M_s f)(s,p) = sf(s,p) \mbox{ and } (M_p f)(s, p) = p f(s,p). \label{MsMp} \end{equation}

\begin{defn} \label{admissibility}

A kernel $k((s,p), (t,q))$ on $\mathbb G$ is called admissible if the pair of multiplication operators $(M_s, M_p)$ on the reproducing kernel Hilbert space $H_k$ is a $\Gamma$-contraction on $H_k$.

\end{defn}

The interpolation result is the following.

\vspace*{5mm}

\noindent \textbf{Interpolation Theorem.} {\em
Given $\lambda_1, \lambda_2, \ldots ,\lambda_n$ in $\mathbb G$ and
$w_1, w_2, \ldots , w_n$ in $\overline{\mathbb D}$, there is a
function $f$ in $H^\infty(\mathbb G)$ with $\| f \|_\infty \le 1$ and satisfying $f(\lambda_i) = w_i, i=1,2, \ldots ,n$
if and only if for every admissible kernel $k$, the matrix
\begin{equation} \label{Pick-condition} \left( \left( \; \; (1 - w_i
\overline{w}_j) k(\lambda_i , \lambda_j) \; \; \right) \right)\end{equation}
is positive semi-definite.}

There is a dual way of stating the Interpolation Theorem in terms of the parametrized co-ordinate functions $\varphi$. These functions are introduced in the next sub-section and the dual statement of the Interpolation Theorem is postponed till the Epilogue section.

\vspace*{5mm}

\subsection{Realization}
The best way to prove an interpolation theorem as above is by exhibiting a realization formula. This is what we do. To get a feel of what a realization formula is all about, we mention the remarkable result that a function $f$ is in $H^\infty(\mathbb D)$ and satisfies $\| f \|_\infty \le 1$ if and only if there is a Hilbert space $H$ and a unitary operator
$$U = \dispmatrix A & B \\ C & D \\ : \mathbb C \oplus H \rightarrow \mathbb C \oplus H$$ such that
$$ f(z) = A + zB (I - zD)^{-1} C.$$
Agler generalized this elegantly to the bidisk in \cite{Agler}. He showed that a function $f$ is in $H^\infty(\mathbb D^2)$ and satisfies $\| f \| \le 1$ if and only if there is a graded Hilbert space $H = H_1 \oplus H_2$ and a unitary operator
$$U = \dispmatrix A & B \\ C & D \\ : \mathbb C \oplus H \rightarrow \mathbb C \oplus H$$ such that writing $P_1$ for the projection from $H$ onto $H_1$ and $P_2$ for the projection from $H$ onto $H_2$, we have
$$ f(z) = A + B (z_1 P_1 + z_2 P_2) (I - D (z_1 P_1 + z_2 P_2) )^{-1} C.$$

In the same tradition, the following theorem is called the Realization Theorem because the fourth item in the list of equivalent statements realizes the function $f$ as a concrete formula. The intermediate steps are interesting in their own right. However, before we state the theorem, we need to introduce a parametrized family of functions. For $\alpha \in \Dbar$ and $(s,p) \in \mathbb G$, let
$$ \varphi(\alpha, s , p) = \frac{2 \alpha p - s}{2 - \alpha s}$$
which is defined for all $(\alpha, s, p)$ satisfying $2 - \alpha s \neq 0$. Since $|s| < 2$ for all $(s,p) \in \mathbb G$, this function is well-defined on $\Dbar \times \mathbb G$. The notation $\varphi ( \alpha, \cdot)$ will mean that for a fixed $\alpha$, we are considering it as a function on $\mathbb G$ and $\varphi( \cdot , s, p)$ will mean that for a fixed $(s,p)$, we are considering it as a function on $\Dbar$. These functions will be called the parametrized co-ordinate functions because Agler and Young proved in Theorem 2.1 of \cite{ay-jga} that
\begin{equation} \label{co-ordinates} (s,p) \in \mathbb G \mbox{ if and only if } \varphi(\alpha, s, p) \in \mathbb D \end{equation} for all $\alpha$ in the closed unit disk. Thus, the membership of $(s,p)$ in $\mathbb G$ is determined by whether the parametrized co-ordinate functions evaluated at $(s,p)$ have modulus less than one for every parameter $\alpha$. We note that for every $\alpha \in \Dbar$, the function $\varphi(\alpha , \cdot)$ is in the norm unit ball of $H^\infty(\mathbb G)$ and for every $(s,p) \in \mathbb G$, the function $\varphi( \cdot , s, p)$ is in $C(\Dbar)$. If we write $\lambda$ for the pair $(s,p)$ which we may sometimes do, for example in the statement of the Interpolation Theorem, then we shall write $\varphi(\cdot , \lambda)$ for $\varphi( \cdot , s, p)$.

We shall also need to consider positive semi definite kernels $\Delta : \mathbb G \times \mathbb G \rightarrow C( \Dbar )^*$ where $C( \Dbar )^*$ is the space of all bounded linear functionals on the Banach space $C( \Dbar )$. These are functions $\Delta((s,p),(t,q))$ on $\mathbb G \times \mathbb G$ which satisfy the property  that
$$ \sum_{i,j=1}^N c_i \overline{c_j} \Delta \left( (s_i, p_i), (s_j,p_j) \right) \left( h_i \overline{h_j} \right) \ge 0$$ for any natural number $N$, any $N$ scalars $c_1, c_2, \ldots , c_N$ and any $N$ functions $h_1, h_2, \cdots ,h_N$. Such positive semi definite kernels have been considered in the literature, often in more generality, see \cite{Ambrozie}, \cite{BBF} and \cite{DM}. If the inequality is strict, then $\Delta$ is called a positive definite kernel.

An example of such a positive semi definite kernel $\Delta$ is easily obtained from a regular Borel measure $\mu$ on $\Dbar$ and a function $\delta : \Dbar \times \mathbb G \times \mathbb G$ which has the property that it is a scalar valued weak kernel on $\mathbb G \times \mathbb G$ for every $\alpha$ in $\Dbar$ and is an $L^1(\mu)$ function for every fixed $(s,p)$ and $(t,q)$ in $\mathbb G$. Define $\Delta$ by
\begin{equation} \label{ExampleofDelta} \Delta\left( (s,p), (t,q) \right) (f) = \int_{{\Dbar}} f(\cdot) \delta(\cdot , (s,p), (t,q) ) d\mu. \end{equation}
We shall see a $\Delta$ of this form in Section 3.

\vspace*{5mm}

\noindent \textbf{Realization theorem.} {\em
 The following are equivalent.

\begin{description}
        \item[H] $f$ is a function in $H^\infty(\mathbb G)$ with $\| f \|_\infty \le 1$.
        \item[M] $ (1 - f(s,p) \overline{f(t,q)}) k( (s,p) , (t,q)) $ is a weak kernel for every admissible kernel $k$.
        \item[D] There is a positive semi definite kernel $\Delta : \mathbb G \times \mathbb G \rightarrow C( \Dbar )^*$ such that
        $$ 1 - f(s,p) \overline{f(t,q)} = \Delta ( (s,p), (t,q)) \big( 1 - \varphi(\cdot, s, p) \overline{\varphi(\cdot , t, q)} \big) .$$
        \item[R] There is a Hilbert space $H$, a unital $*$-representation $\pi : C(\Dbar) \rightarrow B(H)$ and a unitray $V : \mathbb C \oplus H \rightarrow \mathbb C \oplus H$ such that writing $V$ as
            $$ V = \left(
                     \begin{array}{cc}
                       A & B \\
                       C & D \\
                     \end{array}
                   \right)$$
                   we have $f(s,p) = A +  B \pi(\varphi(\cdot, s, p)) \big( I_H -  D \pi(\varphi(\cdot, s, p)) \big)^{-1} C$.
      \end{description}}

      An operator $V$ as above is often called a colligation.

      The logical build-up of the paper is that the Realization Theorem implies the Interpolation Theorem which implies the Extension Theorem. Section 2 contains a natural example of an admissible kernel. It is called the Szego kernel and the associated Hilbert space of holomorphic functions is called the Hardy space of the symmetrized bidisk. We shall see that $H^\infty(\mathbb G)$ is the multiplier algebra of the Hardy space. Section 3 has the proof of the realization theorem. Section 4 proves the Nevanlinna-Pick interpolation, the criteria for which is not in terms of the Szego kernel alone, but involves a whole family of kernels, the admissible ones to be precise. Section 5 proves the extension theorem.

The functions $\varphi$ will play the role of "test functions" in the sense of \cite{BallGH} and \cite{DM} for example, although we call them parametrized co-ordinate functions since they really behave as co-oridinates in the case of the symmetrized bidisk. The idea of using a positive semi-definite kernel taking values in the dual space of a suitable Banach space of functions originated in \cite{Ambrozie}, as far as we could see, and had been of great use in later papers \cite{BallGH}, \cite{DM}. The over-arching idea of using a Hahn-Banach separation argument is due to Agler \cite{Agler} and has been found useful in numerous later papers. Many of the arguments of Sections 3 and 4 of this paper are motivated by the book \cite{AMc-book}.

\section{Hardy Space}

\subsection{The space} The function theory on the symmetrized
bidisk provides us with a natural Hilbert space of holomorphic functions on $\mathbb G$.

\begin{defn}\label{GHardy}

The Hardy space $H^2(\mathbb G)$ of the symmetrized bidisk is the vector space of those holomorphic functions $f$ on $\mathbb G$ which satisfy
$$\mbox{sup}_{\,0<r<1}\int_{\mathbb T \times \mathbb T} |f\circ \pi (r\,e^{i\theta_1}, r\,e^{i\theta_2})|^2
|J(r\,e^{i\theta_1}, r\,e^{i\theta_2})|^2 d\theta_1 d\theta_2 <\infty $$
where $J$ is the complex Jacobian of the
symmetrization map
$$\pi (z_1, z_2) = (z_1+z_2 , z_1z_2)$$
and $d\theta$ is the normalized Lebesgue measure on the unit circle $\mathbb T = \{\alpha
:|\alpha| = 1\}$.
The norm of $f\in H^2(\mathbb G)$ is defined to be
$$\|f\|= \|J\|^{-1}\Big \{\mbox{sup}_{0<r<1}\int_{\mathbb
T \times \mathbb T}|f\circ \pi (r\,e^{i\theta_1}, r\,e^{i\theta_2})|^2 |J(r\,e^{i\theta_1}, r\,e^{i\theta_2})|^2
d\theta_1 d\theta_2 \Big \}^{1/2},$$
where $\|J\|^2={{\int_{\mathbb T \times \mathbb T}}} |J|^2 d\theta_1 d\theta_2.$ \end{defn}

This ensures that the norm of the constant function $1$ is $1$. We used exhaustion above. It is reminiscent of the definition of the Hardy space of the bidisk $H^2(\mathbb D^2)$ which is the collection of holomorphic functions on the bidisk $\mathbb D \times \mathbb D$ such that
$$ \|f\|_{H^2(\mathbb D^2)} \bydef \sup_{0<r<1} \Big( \int_{{\mathbb{T}} \times \mathbb T} | f(r e^{i \theta_1}, r e^{i \theta_2}) |^2 d \theta_1 d\theta_2 \Big)^{1/2}
< \infty.
$$
The relationship between these two Hardy spaces (of the symmetrized bidisk and of the bidisk) goes deeper because of the following lemma from \cite{MSRZ}.

 \begin{lem}
As a Hilbert space, $H^2(\mathbb G)$ is isomorphic to the subspace
$$H^2_{\rm anti} (\mathbb D^2) \bydef \{ f \in H^2(\mathbb D^2) : f(z_1 , z_2) = - f(z_2 , z_1) \}$$
of $H^2(\mathbb D^2)$. The isomorphism is given by $\tau_1(f) = J (f \circ \pi)$. \end{lem}

\begin{proof}
It is straightforward by integration that $\tau_1$ preserves norm. Moreover $\tau_1(f)$ is an anti-symmetric function of $z_1$ and $z_2$, i.e., $\tau_1f(z_1 , z_2) = - \tau_1f(z_2 , z_1)$. This is because of $J$. The value of $J$ at a point $(z_1, z_2)$ is $z_1 - z_2$. So the product of the symmetric function $f \circ \pi$ and the value of the Jacobian $J$ at $(z_1, z_2)$ is anti-symmetric. Let $H^2_{\rm anti}(\mathbb D^2)$ be the closed subspace of
anti-symmetric functions in the Hardy space of the bidisk, i.e.,
$H^2_{\rm anti}(\mathbb D^2)$ is the space of those $f$ in
$H^2(\mathbb D^2)$ which satisfy $f(z_1, z_2) = -f(z_2 , z_1)$. Every anti-symmetric function $g$ in $H^2(\mathbb D^2)$ is divisible by $z_1 - z_2$.
Thus the function $J^{-1}g$ is a symmetric holomorphic function on $\mathbb D^2$. So there is a holomorphic function $f$ on $\mathbb G$ such that $J^{-1} g = f \circ \pi$. This $f$ is in $H^2(\mathbb G)$ because $g = J (f \circ \pi )$ is in $H^2(\mathbb D^2)$ which precisely is the definition of an $H^2(\mathbb G)$ function. So $\tau_1(f) = g$. Consequently, every anti-symmetric function in $H^2(\mathbb D^2)$ is in the range of $\tau_1$. \end{proof}

This identification enables us to show that $H^2(\mathbb G)$ is a reproducing kernel Hilbert space in the sense of Definition \ref{kernel}. We shall use Theorem 3.1 of \cite{MSRZ} to note that $H^2_{\rm anti}(\mathbb D^2)$ has the following reproducing kernel:
$$\frac{(z_1 - z_2)((w_1 - w_2)}{2(1 - z_1 \bar{w}_1)(1 - z_1 \bar{w}_2)(1 - z_2 \bar{w}_1)(1 - z_2 \bar{w}_2)}.$$
Using the isomorphism $\tau_1$ above, a straightforward computation gives the formula for the reproducing kernel of the Hardy space of the symmetrized bidisk. It is
$$ k_S((s_1, p_1), (s_2, p_2)) = \frac{1}{(1 - p_1 \bar{p}_2)^2 - (s_1 - \bar{s}_2 p_1)(\bar{s}_2 - s_1 \bar{p}_2)}.$$
This has the property that $\langle f , k_S(s,p) \rangle = f(s,p)$ for any $f \in H^2(\mathbb G)$ and $(s,p) \in \mathbb G$.

\begin{defn}

The kernel $k_S$ obtained above, i.e., the reproducing kernel for the Hardy space of the symmetrized bidisk is called the Szego kernel of the symmetrized  bidisk \end{defn}

\subsection{The multipliers}
A holomorphic function $\varphi : \mathbb G \rightarrow \mathbb C$ is called a multiplier on $H^2(\mathbb G)$ if the liner transformation $M_\varphi$ on $H^2(\mathbb G)$ defined by
$$ (M_\varphi f)((s,p)) = \varphi((s,p)) f((s,p)), \mbox{ for } (s,p) \in \mathbb G $$
is a bounded linear operator on $H^2(\mathbb G)$. Clearly, the set of all multipliers form an algebra, called the multiplier algebra of $H^2(\mathbb G)$ and denoted by $\clm(H^2(\mathbb G)$.

\begin{lem}
The multiplier algebra $\clm(H^2(\mathbb
G)$ is isometrically isomorphic with $H^\infty(\mathbb G)$. \end{lem}

\begin{proof} From general theory of reproducing kernel Hilbert spaces and multipliers on them, it is well-known that a multiplier $\varphi$ belongs to $H^\infty(\mathbb G)$, the algebra of bounded holomorphic functions on $\mathbb G$. Moreover, if $\| \varphi \|_\infty$ denotes the $H^\infty$-norm
$$ \| \varphi \|_\infty \bydef \sup_{(s,p) \in \mathbb G} |\varphi(s,p)|$$
then $\|M_\varphi \| \ge \| \varphi \|_\infty$, see inequality (2.33) in \cite{AMc-book} for example. In the particular case of the symmetrized bidisk, it is also true that $H^\infty(\mathbb G) \subseteq \clm(H^2(\mathbb G)$ and $\|M_\varphi \| = \| \varphi \|_\infty$. Indeed, for $f \in H^2(\mathbb G)$ and $\varphi \in H^\infty (\mathbb G)$ and $r \in (0,1)$, we have
\begin{eqnarray*}
& & \int_{\mathbb T \times \mathbb T} |\varphi \circ \pi (r\,e^{i\theta_1}, r\,e^{i\theta_2})|^2 |f\circ \pi (r\,e^{i\theta_1}, r\,e^{i\theta_2})|^2
|J(r\,e^{i\theta_1}, r\,e^{i\theta_2})|^2 d\theta_1 d\theta_2 \\
& \le & \| \varphi \|_\infty \int_{\mathbb T \times \mathbb T}
|f\circ \pi (r\,e^{i\theta_1}, r\,e^{i\theta_2})|^2
|J(r\,e^{i\theta_1}, r\,e^{i\theta_2})|^2 d\theta_1 d\theta_2
\end{eqnarray*} which means that $\| M_\varphi f \| \le \| \varphi
\|_\infty \| f \|$. \end{proof}

Multiplication by the co-ordinate functions are often special. On $H^2(\mathbb D^2)$, they are $M_{z_1}$ and $M_{z_2}$. On $H^2(\mathbb G)$, they are $M_{s}$ and $M_{p}$.

\begin{lem}
Szego kernel of the symmetrized bidisk is an admissible kernel. \end{lem}

\begin{proof}
We shall need to show that the operator pair $(M_s, M_p)$ on
$H^2(\mathbb G)$ is a $\Gamma$-contraction. We shall in fact show
that it is a $\Gamma$-isometry by showing that it is unitarily
equivalent to a $\Gamma$-isometry.

The first thing to note is that the symmetrization of a pair of commuting isometries is a $\Gamma$-isometry, see \cite{BPSR}. Thus $(M_{z_1} + M_{z_2}, M_{z_1} M_{z_2})$ is a $\Gamma$-isometry on $H^2 (\mathbb D^2)$, Secondly, the subspace $H^2_{\rm anti}(\mathbb D^2)$ is
invariant under this $\Gamma$-isometry. Restriction of a $\Gamma$-isometry to an invariant subspace  is a $\Gamma$-isometry, see \cite{BPSR}.  Thus the commuting pair $$((M_{z_1} + M_{z_2})|_{H^2_{\rm
anti}(\mathbb D^2)} , (M_{z_1} M_{z_2})|_{H^2_{\rm anti}(\mathbb
D^2)})$$ is a $\Gamma$-isometry.
We shall show unitary equivalence of the operator pair $(M_s, M_p)$
with the pair $$( (M_{z_1} + M_{z_2})|_{H^2_{\rm anti}(\mathbb
D^2)}, (M_{z_1} M_{z_2})|_{H^2_{\rm anti}(\mathbb D^2)}).$$ We
note that by definition of the unitary $\tau_1$, we have that for any
$f \in H^2(\mathbb G)$,
$$((\tau_1 M_s)f) (z_1, z_2) = \| J \|^{-1} J(z_1, z_2) (z_1 +
z_2) f(\pi(z_1, z_2))$$ and
$$(M_{z_1 + z_2} \tau_1f)(z_1, z_2) = (z_1 + z_2)\| J \|^{-1} J(z_1, z_2) f(\pi(z_1, z_2)).$$
Thus $\tau_1 M_s = M_{z_1 + z_2} \tau_1$. Similarly, for $M_p$ and
$(M_{z_1} M_{z_2})|_{H^2_{\rm anti}(\mathbb D^2)}$. That completes
the proof.
\end{proof}

\section{Proof of the Realization Theorem}

This is the largest section of this paper because the proof of the Realization Theorem will use various concepts. The first step of \textbf{(H)} implying \textbf{(M)} is accomplished in the following lemma. The ideas and the arguments in the proof of this lemma are adapted from \cite{ColeWermer} where the authors proved a similar lemma for the disk algebra.

\begin{lem} If $f$ is a function in $H^\infty(\mathbb G)$ with $\| f \|_\infty \le 1$, then
        $$ (1 - f(s,p) \overline{f(t,q)}) k( (s,p) , (t,q)) $$ is a weak kernel for every admissible kernel $k$.
\label{HimpliesM}
\end{lem}

\begin{proof} Let $n$ be a positive integer and let $\lambda_1 = (s_1, p_1), \lambda_2 = (s_2,p_2), \ldots , \lambda_n=(s_n,p_n)$ be $n$ points in $\mathbb G$. We shall need to show that the $n \times n$ matrix
$$ \big( (1 - f(s_i,p_i) \overline{f(s_j,p_j)}) k( (s_i,p_i) , (s_j,p_j)) $$
is a positive semi-definite for every admissible kernel $k$. Let $w_i = f(\lambda_i)$. Let $w = (w_1, w_2, \ldots ,w_n) $. We define two interpolation sets. The first one is the Pick body.
$$ \mathcal D(\lambda) = \{ w : \mbox{ there is } f \in H^\infty (\mathbb G) \mbox{ with } f(\lambda_j) = w_j \mbox{ for } 1 \le j \le n \mbox{ and } \| f \|_\infty \le 1\}.$$
The other is the interpolation set associated with the algebra $A(\Gamma$).
$$ \mathcal D = \{ w : \mbox{ for all } \epsilon > 0, \mbox{ there is } f \in A(\Gamma) \mbox{ with } f(\lambda_j) = w_j \mbox{ for } 1 \le j \le n \mbox{ and } \| f \| \le 1 + \epsilon \}.$$

Note that $$\mathcal D = \cap_{\epsilon > 0} A_\epsilon,$$ where $$A_\epsilon = \{ w : \mbox{ there is } f \in A(\Gamma) \mbox{ with } f(\lambda_j) = w_j \mbox{ for } 1 \le j \le n \mbox{ and } \| f \| \le 1 + \epsilon \}.$$ The linear map $L : A(\Gamma) \rightarrow \mathbb C^n$ sending $f \rightarrow (f(\lambda_1), f(\lambda_2), \ldots , f(\lambda_n))$ is a continuous surjection and hence is an open map proving that the complement of $A_\epsilon$ is open. Thus $\mathcal D$ is closed.

The two sets $\mathcal D(\lambda)$ and $\mathcal D$ are same. Indeed, if $w \in \mathcal D(\lambda)$, then there is an $f \in H^\infty (\mathbb G) \mbox{ with } f(\lambda_j) = w_j \mbox{ for } 1 \le j \le n \mbox{ and } \| f \|_\infty \le 1$. For each $r \in (0,1)$, the function $f_r(s,p) = f(rs, r^2p)$ is in the unit ball of the algebra $A(\Gamma)$.
So the points $\{ (f(r\lambda_1), f(r\lambda_2), \ldots , f(r\lambda_n)) \}_{0 < r < 1}$ are all in $\mathcal D$. Since $\mathcal D$ is closed and these points converge to $w$ as $r \rightarrow 1$, we have $w \in \mathcal D$.

Conversely, if $w \in \mathcal D$, then for every $m \ge 1$, there is a function $f_m$ in $A(\Gamma)$ such that $f_m(\lambda_j) = w_j$ for $1 \le j \le n$ and $\| f_m \| \le 1 + \frac{1}{m}$. Now a normal family argument implies that there is a subsequence $\{ m_k\}_{k \ge 1}$ such that $\{ f_{m_k} \}$ converges uniformly on compact subsets and hence $f = \lim f_{m_k}$ is holomorphic in $\mathbb G$. It has norm no greater than one. Thus $w \in \mathcal D$.

Given that there is a function $f$ in $H^\infty(\mathbb G)$ with $\| f \|_\infty \le 1$ and satisfying $f(\lambda_i) = w_i, i=1,2, \ldots ,n$, we want to show that the $n\times n$ matrix in (\ref{Pick-condition}) is positive definite. By what we did above, we can assume $f$ to be in $A(\Gamma)$. First suppose $f$ is a polynomial. The general case will be dealt
with in a short while using a sequence of polynomials. Let
$$f (s,p) = \sum_{k,l = 1}^N a_{kl} s^k p^l.$$
Let $\check{f}$ be another polynomial
$$\check{f} (s,p) = \sum_{k,l = 1}^N \overline{a}_{kl} s^k p^l.$$

Let $k$ be an admissible kernel. Then $(M_s, M_p)$ forms a
$\Gamma$-contraction on $H_k$. Let $\lambda_i = (s_i , p_i)$ for
$i=1,2, \ldots , n$. Let $k_j$ be the kernel function $$ k_j(z) =
k(z , \lambda_j) \mbox{ for } j=1,2, \ldots ,n.$$ Let $\mathcal L$
be the $n$-dimensional space spanned by $k_1, k_2, \ldots ,k_n$.
Define operators $T_1$ and $T_2$ on $\mathcal L$ by
$$ T_1^* k_j = \overline{s}_j k_j \mbox{ and } T_2^* k_j = \overline{p}_j
k_j.$$ It is straightforward that $\mathcal L$ is an invariant subspace for the $\Gamma$-contraction $(M_s^*, M_p^*)$ and
$$ M_s^*|_{\mathcal L} = T_1^* \mbox{ and } M_p^*|_{\mathcal L} =
T_2^*.$$ Thus $(T_1^*, T_2^*)$ is a $\Gamma$-contraction. So $\|
\check{f}(T_1^*, T_2^*) \| \le 1$. Now, note that
$$\check{f}(T_1^*, T_2^*) k_j = \sum_{k,l = 1}^N \overline{a}_{kl}
\overline{s_j}^k \overline{p_j}^l k_j = \overline{f(\lambda_j)}
k_j = \overline{w}_j k_j.$$ Now, a straightforward computation
shows that contractivity of $\check{f}(T_1^*, T_2^*)$ is
equivalent to the matrix (\ref{Pick-condition}) being positive
definite.

For a general $f$ in $A(\Gamma)$, we can get a
sequence of polymonials $\{ p_m \}$ that converges uniformly over
$\Gamma$ to $f$ by Oka-Weil theorem. What we proved above tells us
that for every $m$, the matrix
$$ \left( \left( \; \; (1 - p_m(\lambda_i)
\overline{p_m(\lambda_j)}) k(\lambda_i , \lambda_j) \; \; \right)
\right)$$ is positive definite. This matrix converges to the
matrix in (\ref{Pick-condition}) as $m$ tends to infinity. Since
the set of positive definite matrices is closed, we are done.
\end{proof}

   There is an elegant characterization of a $\Gamma$-contraction. Agler and Young showed that a pair $(S,P)$ of commuting bounded operators is a $\Gamma$-contraction (Definition \ref{Gamma-contraction}) if and only if $(2\alpha P - S) (2 - \alpha S)^{-1} $ is a contraction for every $\alpha \in \mathbb D$, see Theorem 1.5 in \cite{ay-jot}. Applying this criterion to reproducing kernel Hilbert spaces, we get the following.

\begin{lem} \label{admi}
A kernel $k$ on $\mathbb G$ is admissible if and only if
$$ (1 - \varphi(\alpha, s, p) \overline{\varphi(\alpha, t, q)}) k\big( (s,p) , (t,q) \big)$$
is a weak kernel for every $\alpha \in \overline{\mathbb D}$. \end{lem}

\begin{proof} Let $k$ be a kernel on $\mathbb G \times \mathbb G$ and let $H_k$ be the corresponding reproducing kernel Hilbert space.

The kernel $k$ is admissible \\
$\Rightarrow$ The operator pair $(M_s, M_p)$ is a $\Gamma$-contraction on $H_k$ \\
$\Rightarrow$ The operator  $(2 \alpha M_p - M_s)(2 - \alpha M_s)^{-1}$ is a contraction for every $\alpha \in \mathbb D$ (by part (v) of Theorem 1.5 of \cite{ay-jot}) \\
$\Rightarrow$ $ (1 - \varphi(\alpha, s, p) \overline{\varphi(\alpha, t, q)}) k\big( (s,p) , (t,q) \big)$
is a weak kernel for every $\alpha \in {\mathbb D}$ (by using the fact that $M_\psi^* k ( \cdot , (t,q)) = \overline{\psi(t,q)} k ( \cdot , (t,q))$ for every multiplication operator $M_\psi$ on the space $H_k$).

We can extend the result to all of $\Dbar$ because $\varphi$ is a continuous function of $\alpha$. Conversely, if we know that $ (1 - \varphi(\alpha, s, p) \overline{\varphi(\alpha, t, q)}) k\big( (s,p) , (t,q) \big)$
is a weak kernel for every $\alpha \in \overline{\mathbb D}$, then by putting $\alpha = 0$, we obtain that $M_s$ is a bounded operator. By putting $\alpha = 1$ and $\alpha = -1$ and adding, we get that $M_p$ is a bounded operator. Now, all the steps in the above implications can be reversed because part (v) of Theorem 1.5 of \cite{ay-jot} is a characterization.
 \end{proof}

Taking cue from the lemma above, it is now natural to make the following definition.

\begin{defn} \label{AdmissibilityonY}
A weak kernel $k$ on a subset $Y$ of $\mathbb G$ is called admissible if
$(1 - \varphi(\alpha, s, p) \overline{\varphi(\alpha, t, q)}) k\big( (s,p) , (t,q) \big)$
is a weak kernel on $Y \times Y$ for every $\alpha \in \overline{\mathbb D}$. \end{defn}

The following lemma decomposes $ 1 - f(s,p) \overline{f(t,q)}$ and hence accomplishes the step of \textbf{(M)} implying \textbf{(D)} of the Realization Theorem.

\begin{lem}[Decomposition]\label{MimpliesD}
If $ (1 - f(s,p) \overline{f(t,q)}) k( (s,p) , (t,q)) $ is a weak kernel for every admissible kernel $k$, then there is a positive semi definite kernel  $\Delta$ defined on $\mathbb G \times \mathbb G$ and taking values in $C(\Dbar)^*$ such that
        $$ 1 - f(s,p) \overline{f(t,q)} = \Delta\big( ( s, p) , (t, q) \big) \left( 1 - \varphi(\cdot , s , p) \overline{\varphi(\cdot , t, q)} \right).$$ \end{lem}

\begin{proof}

\textbf{1. The family of kernels} $\mathbf{B_\alpha}$
The first step in this proof is the fact that for every $\alpha$ in the closed unit disk $\overline{\mathbb D}$,
$$ B \big( \alpha , (s,p), (t,q) \big) = \frac{1}{1 - \varphi(\alpha, s, p) \overline{\varphi(\alpha, t , q)}} $$
is a kernel on $\mathbb G \times \mathbb G$.

By using the characterization (\ref{co-ordinates}) and by using the Szego kernel
 $$ (z,w) \rightarrow \frac{1}{1 - z\overline{w}}$$
 of the open unit disk, it is obvious that $B \big( \alpha , \cdot , \cdot \big)$ is a kernel.

\textbf{2. The closed wedge} Let $Y \subset \mathbb G$ be a finite set. Let the cardinality of $Y$ be $N$. Its elements are $\lambda_1, \lambda_2, \ldots ,\lambda_N$. We shall need to use the co-ordinates of $\lambda_i$. So let $\lambda_i = (s_i , p_i)$. Let
 \begin{eqnarray*} \mathcal T_Y & = & \{ \Delta\big( ( s, p) , (t, q) \big) \left( 1 - \varphi(\cdot , s , p) \overline{\varphi(\cdot , t, q)} \right) \\
 & : &  \Delta \mbox{ is a } C(\Dbar)^* \mbox{ valued positive semi definite function  on } \mathbb G \times \mathbb G\}, \end{eqnarray*}
i.e., a member of $ \mathcal T_Y$ is an $N \times N$ matrix whose $(i,j)$th. entry is of the form
\begin{equation} \label{TYelement}  \Delta ((s_i,p_i) , (s_j,p_j)) \left( 1 - \varphi(\cdot , s_i , p_i) \overline{\varphi(\cdot , s_j, p_j)} \right) \end{equation}
  Clearly, $ \mathcal T_Y$ is a wedge in the set of $N \times N$ self-adjoint matrices, i.e., $ \mathcal T_Y$ is convex and when an element is multiplied by a non-negative real number, it remains in the set.

We use the fact that $B_\alpha$ is a kernel to show that the kernel $\mathbf{1} ((s_i, p_i) , (s_j , p_j)) = 1$ for all $i,j = 1,2, \ldots ,N$ is in $\mathcal T_Y$. To that end, consider a probability measure $\mu$ on $\Dbar$ and define a particular positive semi definite function $\Delta$ by
$$ \Delta \big( (s,p), (t,q) \big) f = \int_{{\Dbar}} B \big( \alpha , (s,p), (t,q) \big) f(\alpha) d\mu(\alpha).$$
Indeed, $\mathbf{1}$ is $ \Delta \big( (s,p), (t,q) \big) \left( 1 - \varphi(\cdot , s , p) \overline{\varphi(\cdot , t, q)} \right) $ and hence is in $\mathcal T_Y$ by definition of $\mathcal T_Y$. It is clear from the definition of $\mathcal T_Y$ that the Schur product of an element of $\mathcal T_Y$ and a non-negative definite $N \times N$ matrix is in $\mathcal T_Y$. Thus any non-negative definite $N \times N$ matrix, being the Schur product of itself and $\mathbf{1}$  is in $\mathcal T_Y$. In particular, the rank one kernel $(( c_i \bar{c_j} ))$ is in $\mathcal T_Y$.

We would like to show that the set $ \mathcal T_Y$ is closed. Consider a sequence $\Delta_n$ such that $\Delta_n ((s_i,p_i) , (s_j,p_j)) \left( 1 - \varphi(\cdot , s_i , p_i) \overline{\varphi(\cdot , s_j, p_j)} \right)$ converges to an $N \times N$ matrix $A = (( a_{ij} ))$. We need to show that $A$ is in $ \mathcal T_Y$. We use the fact that for any $(s,p) \in \mathbb G$, $\sup\{ | \varphi (\alpha , s, p) | : \alpha \in \Dbar \} < 1$.
Thus there is an $\epsilon > 0$ such that $ 1 - | \varphi (\alpha , s, p) |^2 > \epsilon$ for all $\alpha \in \Dbar$. Since $\Delta_n ((s_i,p_i) , (s_i,p_i))$ is a positive linear functional, we have
$$ \Delta_n ((s_i,p_i) , (s_i,p_i)) \big( 1 - | \varphi (\alpha , s_i, p_i) |^2 \big) > \Delta_n ((s_i,p_i) , (s_i,p_i)) (\epsilon 1) = \epsilon \| \Delta_n ((s_i,p_i) , (s_i,p_i)) \|.$$
Taking limit, we obtain $a_{ii} > \epsilon  \| \Delta_n ((s_i,p_i) , (s_i,p_i)) \|.$ Finiteness of the set $Y$ guarantees that we get a single $\epsilon$ serving for all $i$. Then the fact that each $\Delta_n$ is a positive semi definite function ensures that each $\Delta_n ((s_i,p_i) , (s_j,p_j)) $ is uniformly norm bounded in $n$. Using weak $*$-compactness and again finiteness of the set $Y$, we get a subsequence $\{ n_l \}$ such that $\Delta_{n_l} ((s_i,p_i) , (s_j,p_j))$ is convergent, to  $\Delta ((s_i,p_i) , (s_j,p_j))$ say, for every $i$ and $j$. Thus,
\begin{eqnarray*} a_{ij} & = & \lim_{n_l \rightarrow \infty } \Delta_{n_l} ((s_i,p_i) , (s_j,p_j)) \left( 1 - \varphi(\cdot , s_i , p_i) \overline{\varphi(\cdot , s_j, p_j)} \right)\\
 & = & \Delta ((s_i,p_i) , (s_j,p_j)) \left( 1 - \varphi(\cdot , s_i , p_i) \overline{\varphi(\cdot , s_j, p_j)} \right) \end{eqnarray*}
which is what was required to be shown for membership of $A$ in the wedge $ \mathcal T_Y$. So $ \mathcal T_Y$ is closed.

\textbf{3. The Hahn-Banach functional} Define $g : \mathbb G \times \mathbb G \rightarrow \mathbb C$ by
$$ g\big( (s,p), (t,q) \big) = 1 - f(s,p) \overline{f(t,q)}.$$
Then $g$ is self-adjoint (i.e., $g \big( (s,p) , (t,q) \big) = \overline{g \big( (t,q) , (s,p) \big)} $) and satisfies
$$ g \cdot k : \big((s,p) , (t,q)\big) \rightarrow g\big( (s,p) , (t,q) \big) k\big( (s,p) , (t,q) \big)$$
is positive semi definite for every admissible kernel $k$. We want to show that the restriction $h$ of $g$ to $ Y \times Y$ is in $\mathcal T_Y$.

If not, then by Hahn-Banach Theorem, there is a real linear functional $L$ on $N \times N$ matrices which is non-negative on $\mathcal T_Y$, but is negative when evaluated at $h$. This linear functional can be assumed to be of the form $L(T) =$ tr$TK^t$ for some self-adjoint matrix $K$. Now,
$$\mbox{ tr }TK^t = \sum_{i=1}^{N} \sum_{j=1}^{N} t_{il} k_{il}.$$
In particular, if $T = (( c_i \overline{c_j} ))$ which is in $\mathcal T_Y$, then $\sum_{i=1}^{N} \sum_{j=1}^{N} c_i \overline{c_l} k_{il} \ge 0$. This shows that $K$ is a positive semi-definite $N \times N$ matrix.  We shall denote its extension to $\mathbb G$ as a weak kernel by $K$ as well. This weak kernel is admissible on $Y$. Indeed, all that needs to be shown is that
$$ (1 - \varphi(\alpha, s, p) \overline{\varphi(\alpha, t, q)}) K \big( (s,p) , (t,q) \big)$$
is a weak kernel on $Y$ for every $\alpha \in \Dbar$. This is true because
 \begin{eqnarray*} & & \sum_{i,j=1}^n c_i \overline{c_j} \left( (1 - \varphi(\alpha, s_i, p_i) \overline{\varphi(\alpha, s_j, p_j)}) K\big( (s_i,p_i) , (s_j,p_j) \big) \right) \\
 & = & L \big(c_i \overline{c_j} \left( (1 - \varphi(\alpha, s_i, p_i) \overline{\varphi(\alpha, s_j, p_j)}) \big) \right) \end{eqnarray*}
and $L$ is non-negative on elements of $\mathcal T_Y$.

Now, let $K_1$ be any admissible kernel. Then $K + \epsilon K_1$ is an admissible kernel for any $\epsilon > 0$. By hypothesis, $g \cdot ( K + \epsilon K_1 ) $ is a weak kernel because the Schur product of $g$ with any admissible kernel is positive   semi definite. Since this is true for every $\epsilon > 0$, we have $g \cdot K$ is a weak kernel. But that means that $L(h) \ge 0$ which contradicts the way $L$ has been chosen.

Now a standard application of Kurosh's lemma completes the proof of the Decomposition.
\end{proof}

To go to the next step of proof of the Realization Theorem, we now prove the step \textbf{(D)} implies \textbf{(R)} of the theorem. For that, we first need a preparatory lemma.

\begin{lem} \label{GNS} If $\Delta : \mathbb G \times \mathbb G \rightarrow C(\Dbar)^*$ is a positive semi definite kernel, then there is a Hilbert space $H$ and
a function $L : \mathbb G \rightarrow B(C(\Dbar) , H)$ such that $$\Delta \left((s,p), (t,q) \right) (h_1 \overline{h_2}) =
\langle L(s,p) h_1, L(t,q)h_2 \rangle_H$$ for all $h_1, h_2 \in C(\Dbar)$ and $(s,p), (t,q) \in \mathbb G$.

Moreover, there is a unital $*$-representation $\pi : C(\Dbar) \rightarrow B(H)$ such that $L(s,p) h_1h_2 = \pi(h_1) L(s,p) h_2$.

\end{lem}

\begin{proof} The construction is standard and hence we only sketch and refer the reader to \cite{Ambrozie} and \cite{DM} for details. Let $V$ denote the vector space with basis $\{ (s,p) \in \mathbb G \}$. Take $V \otimes C(\Dbar)$ with the positive semidefinite sesquilinear form on elementary tensors as
$$ \langle (s,,p) \otimes h_1, (t,q) \otimes h_2 \rangle = \Delta \left((s,p), (t,q) \right) (h_1 \overline{h_2})$$
and then extending it by linearity. We quotient by the null space of this form and complete to get $H$. Define $L(s,p) h = (s,p) \otimes h$ for $h \in C(\Dbar)$ and extend linearly. The representation $\pi$ is defined by $\pi(h)((s,p) \otimes h^\prime) = (s,p) \otimes h h^\prime$. \end{proof}

\begin{lem}[Realization]
If there  is a positive semi definite kernel $\Delta : \mathbb G \times \mathbb G \rightarrow C( \Dbar )^*$ such that
        $$ 1 - f(s,p) \overline{f(t,q)} = \Delta ( (s,p), (t,q)) \big( 1 - \varphi(\cdot, s, p) \overline{\varphi(\cdot , t, q)} \big),$$
        then there is a Hilbert space $H$, a unital $*$-representation $\pi : C(\Dbar) \rightarrow B(H)$ and an isometry $V : \mathbb C \oplus H \rightarrow \mathbb C \oplus H$ such that writing $V$ as
            $$ V = \left(
                     \begin{array}{cc}
                       A & B \\
                       C & D \\
                     \end{array}
                   \right)$$
                   we have $f(s,p) = A +  B \pi(\varphi(\cdot, s, p)) \big( I_H -  D \pi(\varphi(\cdot, s, p)) \big)^{-1} C$.

 \end{lem}

 \begin{proof}
 We rewrite the given condition as
 $$ 1 + \Delta ( (s,p), (t,q)) \big( \varphi(\cdot, s, p) \overline{\varphi(\cdot , t, q)} \big) = f(s,p) \overline{f(t,q)} + \Delta ( (s,p), (t,q)) \big( 1  \big).$$
 By the lemma above, there is a Hilbert space $H$ and a function $L : \mathbb G \times \mathbb G \rightarrow B(C(\Dbar) , H)$ such that
 $$\Delta \left((s,p), (t,q) \right) (h_1 \overline{h_2}) =
\langle L(s,p) h_1, L(t,q)h_2 \rangle_H$$ for all $h_1, h_2 \in C(\Dbar)$ and $(s,p), (t,q) \in \mathbb G$. Hence,
 $$ 1 +  \langle L (s,p) \varphi(\cdot , s, p), L (t,q) \varphi(\cdot , t, q) \rangle = f(s,p) \overline{f(t,q)} + \langle L (s,p) 1 , L (t,q) 1 \rangle.$$
 By virtue of the representation $\pi$ obtained in the lemma above, this is the same as
 $$ 1 +   \langle \pi \varphi(\cdot , s, p) L (s,p) 1,  \pi \varphi(\cdot , t, q) L (t,q) 1 \rangle = f(s,p) \overline{f(t,q)} + \langle L (s,p) 1 , L (t,q) 1 \rangle.$$
 Now we can define an isometry $V$ from the span of $1 \oplus \pi \varphi(\cdot , s, p) L (s,p) 1 : (s,p) \in \mathbb G$ into the span of $f(s,p) \oplus L (s,p) 1$ such that
 $$ V \left( \begin{array}{c} 1 \\ \pi \varphi(\cdot , s, p) L (s,p) 1 \end{array} \right) = \left( \begin{array}{c} f(s,p) \\ L (s,p) 1 \end{array} \right)$$
 and then extending by linearity. By a standard technique of adding an infinite dimensional Hilbert space to $H$, if required, we can extend $V$ to an isometry from $\mathbb C \oplus H$ into itself. Now, writing $V$ as $\textmatrix A&B\\C&D\\$ as an operator on $\mathbb C \oplus H$, we get
 \begin{eqnarray}
 A + B \pi \varphi(\cdot , s, p) L (s,p) 1 & = & f (s,p) \mbox{ and} \label{formulaforphi} \\
 C + D \pi \varphi(\cdot , s, p) L (s,p) 1 & = & L(s,p) 1. \nonumber \end{eqnarray}
 The second equation above gives that
$$ L(s,p) 1  =  (I_H - D \pi \varphi(\cdot , s, p) )^{-1} C  $$
and hence from (\ref{formulaforphi}), we get
$$ f(s,p)  =  A + B \pi \varphi(\cdot , s, p) (I_H - D \pi \varphi(\cdot , s, p) )^{-1} C . $$
 \end{proof}

 With the above lemma, we have completed the proof of \textbf{(D)} implies \textbf{(R)}. We shall end this section with the following lemma which completes the proof of the equivalences stated in the Realization Theorem by showing that \textbf{(R)} implies \textbf{(H)}.

\begin{lem} Consider a Hilbert space $H$, a unital $*$-representation of $C(\Dbar)$ on $H$ and an isometry $V : \mathbb C \oplus H \rightarrow \mathbb C \oplus H$ of the following form
            $$ V = \left(
                     \begin{array}{cc}
                       A & B \\
                       C & D \\
                     \end{array}
                   \right).$$
                   Then the function $$f(s,p) = A + B \pi \varphi(\cdot , s, p) (I_H - D \pi \varphi(\cdot , s, p) )^{-1} C$$ is in the closed unit ball of $H^\infty(\mathbb G)$.

 \end{lem}

 \begin{proof}
 We would like to note that if we write $z(s,p)$ for $\pi \varphi(\cdot , s, p)$, then the formula above translates to
 $ f(s,p) = A + z(s,p) B ( I - z(s,p) D)^{-1} C$. Since $\pi$ is a unital $*$-representation, $\| \pi \| = 1$ and hence $z(s,p)$ is in the open unit disk because
 $$ | z(s,p) | \le \| \pi \|  \sup\{ | \varphi (\alpha , s, p) | : \alpha \in \Dbar \} < 1.$$
  It is well-known then that such a function has modulus no greater than $1$, see Chapter 6 of \cite{AMc-book} for example. \end{proof}

\section{Proof of the Interpolation Theorem}

If there is an $f$ in the norm unit ball of $H^\infty(\mathbb G)$ interpolating the data, then the positive definiteness of the matrix (\ref{Pick-condition}) can be obtained with the same argument as in the proof of Lemma \ref{HimpliesM}.

Conversely, if the matrix \ref{Pick-condition} is non-negative definite, let its rank be $M$. As in the proof of Lemma \ref{MimpliesD}, we get a positive semi definite kernel  $\Delta$ defined on $\mathbb G \times \mathbb G$ and taking values in $C(\Dbar)^*$ such that
        $$ 1 - w_i \overline{w_j} = \Delta\big( ( s_i, p_i) , (s_j, p_j) \big) \left( 1 - \varphi(\cdot , s_i , p_i) \overline{\varphi(\cdot , s_j, p_j)} \right)$$
        for $i,j=1,2, \ldots , n$. Decomposing $\Delta$ according to Lemma \ref{GNS}, we get
        $$ 1 +  \langle L (s_i,p_i) \varphi(\cdot , s_i, p_i), L (s_j,p_j) \varphi(\cdot , s_j, p_j) \rangle = w_i \overline{w_j} + \langle L (s_i,p_i) 1 , L (s_j,p_j) 1 \rangle $$
        and consequently with $\pi$ as in Lemma \ref{GNS}, we have
         $$ 1 +   \langle \pi \varphi(\cdot , s_i, p_i) L (s_i,p_i) 1,  \pi \varphi(\cdot , s_j, p_j) L (s_j,p_j) 1 \rangle = w_i \overline{w_j} + \langle L (s_i,p_i) 1 , L (s_j,p_j) 1 \rangle.$$

This allows us to define an isometry from the span of $1 \oplus \pi \varphi(\cdot , s_i, p_i) L (s_i,p_i) 1 : i=1, 2, \ldots ,n$ into the span of $w_i \oplus L (s_i,p_i) 1 : i=1, 2, \ldots ,n$ such that
\begin{equation} V \left( \begin{array}{c} 1 \\ \pi \varphi(\cdot , s_i, p_i) L (s_i,p_i) 1 \end{array} \right) = \left( \begin{array}{c} w_i \\ L (s_i,p_i) 1 \end{array} \right) , i=1,2, \ldots ,n. \label{V} \end{equation}
Now the span of all of $\pi \varphi(\cdot , s_i, p_i) L (s_i,p_i) 1 $ for $i=1,2, \ldots ,n$, can be at most of dimension $M$. If it is so, then we have gotten an isometry from $\mathbb C \oplus \mathbb C^M$ to itself. If the dimension falls short of $M$, then we can extend $V$ to an isometry on $\mathbb C \oplus \mathbb C^M$. Writing $V$ as $\textmatrix A & B\\ C & D\\$, let us denote by $f$ the function
$$f(s,p) = A + B \pi \varphi(\cdot , s, p) (I_H - D \pi \varphi(\cdot , s, p) )^{-1} C.$$
By the Realization Theorem of the previous section, this $f$ is indeed a function in the unit norm ball of $H^\infty(\mathbb G)$. Does this $f$ interpolate the data? The answer is yes because of (\ref{V}). That equation tells us
$$ A + B \pi \varphi(\cdot , s_i, p_i) L (s_i,p_i) 1  = w_i \mbox{ and }
 C + D \pi \varphi(\cdot , s_i, p_i) L (s_i,p_i) 1  =  L(s_i,p_i) 1.$$
for every $i=1,2, \ldots ,n.$ Now a straightforward elimination of the values of $L(s_i,p_i) 1$ from the two equations above gives us
$$ w_i = A + B \pi \varphi(\cdot , s_i, p_i) (I_H - D \pi \varphi(\cdot , s_i, p_i) )^{-1} C.$$
for every $i=1,2, \ldots ,n.$ This is what we needed.

\section{Proof of the Extension Theorem}
If $V$ has the $\mathcal A$-extension property, then given any $f \in \mathcal A$, we first get hold of a norm preserving extension $g$, i.e., a function $g \in H^\infty(\mathbb G)$ such that
$$ g|_V = f \mbox{ and } \sup_{\mathbb G} |g| = \sup_V |f|. $$
If $(S,P)$ is a $\Gamma$-contraction subordinate to $V$, then by definition, $f(S,P) = g(S,P)$ and hence
\begin{eqnarray*}
\| f(S,P) \| & = & \| g(S,P) \| \\
& \le & \sup \{ |g(s,p)| : (s,p) \in \mathbb G\} \mbox{ because } (S,P) \mbox{ is a } \Gamma-\mbox{contraction} \\
& = & \sup \{ |f(s,p)| : (s,p) \in \mathbb G\} \mbox{ because } g \mbox{ is a norm-preserving extension}. \end{eqnarray*}
Hence $V$ is an $\mathcal A$-von Neumann set. It is the converse which shows a beautiful interplay of a classical extremal problem with Hilbert space operator theory.

We begin by rephrasing the Interpolation Theorem in such a way that will be useful for proving the Extension Theorem. Let $\lambda_1 = (s_1, p_1), \lambda_2=(s_2, p_2), \ldots ,\lambda_n=(s_n, p_n)$ be $n$ distinct points in the symmetrized bidisk $\mathbb G$. We shall denote this data by $\lambda$. Given any $n \times n$ strictly positive definite matrix $ ( k(i,j) )$, we denote by $k(\cdot , j)$ the vector in $\mathbb C^n$ whose $i$th. entry is $k(i,j)$. We define a pair of commuting bounded operators $X_1$ and $X_2$ on the $n$-dimensional space $\mathcal H$ spanned by $k(\cdot , j), j=1, 2, \ldots ,n$ by
$$ X_1^* k(\cdot , j) = \bar{s}_j k (\cdot , j) \mbox{ and } X_2^* k(\cdot , j) = \bar{p}_j k (\cdot , j).$$  We shall denote by $\mathcal K_\lambda$ the set of all $n \times n$ strictly positive definite matrices $ ( k(i,j) )$ such that $k(i,i) = 1$ and the operator pair $(X_1, X_2)$ forms a $\Gamma$-contraction, i.e., the following holds:
\begin{eqnarray}
(4-s_i\bar{s_j})k(i,j) \geq 0,\;\;\;
(1-p_i\bar{p_j})k(i,j)\geq 0 \;\;\text{ and } \;\;\;\;\;\;\;\;\;\;\nonumber\\\label{admiss conds}
\big( (2 - \alpha s_i) \overline{(2 - \alpha s_j)} - (2 \alpha p_i - s_i) \overline{(2 \alpha p_j - s_j)} \big) k\big( i,j \big) \geq 0.
\end{eqnarray}
Let $w_1, w_2, \ldots ,w_n$ be $n$ points in $\overline{\mathbb{D}}$. The interpolation theorem can be rephrased as follows.

\begin{thm} \label{Interpolation-Reformulation}

There is a function $f \in H^\infty (\mathbb G)$ with $\| f \| \le 1$ and satisfying $f(\lambda_i) = w_i$ for $i=1,2, \ldots ,n$ if and only if the $n \times n$ matrix $ (1 - w_i \bar{w}_j) k(i,j) $ is positive definite for all $k \in \mathcal K_\lambda$.

\end{thm}

We look at a duality argument for the following classical extremal problem. Let $\lambda_1 = (s_1, p_1), \lambda_2=(s_2, p_2), \ldots ,\lambda_n=(s_n, p_n)$ be $n$ distinct points in the symmetrized bidisk $\mathbb G$. Let $w_1, w_2, \ldots ,w_n$ be $n$ points in $\overline{\mathbb{D}}$. A normal family argument shows that the following infimum is attained.

 \begin{eqnarray}\label{extremal} \rho & = & \inf \{ \| f \|_\infty : f \mbox{ is a holomorphic function from } \mathbb G \mbox{ into } \overline{\mathbb{D}} \nonumber \\
 & & \mbox{ satisfying } f(s_i , p_i) = w_i \mbox{ for } i=1,2, \ldots , n\}.\end{eqnarray}

A function $f$ is called {\em {extremal}} if the infimum above is attained for $f$.
\begin{lem}

If $f$ is an extremal for $\rho$, then there is a $\Gamma$-contraction $(S,P)$ subordinate to $\{(s_1,p_1), (s_2,p_2), \ldots ,(s_n,p_n)\}$ such that $\| f(S,P) \| = \rho$. \end{lem}
\begin{proof}
Before proving the lemma let us first prove an interesting property of the set $\mathcal{K}_\lambda$. Note that $\mathcal{K}_\lambda$ is a subset of $n \times n$ self-adjoint complex matrices. We claim that $\mathcal{K}_\lambda$ is closed and bounded with respect to the matrix norm and hence $\mathcal{K}_\lambda$ is compact. That $\mathcal{K}_\lambda$ is bounded is easy, as for all $k \in \mathcal{K}_\lambda$, we have
$$
\|k\|\leq \tr(k)=n \text{ \;\;(since $k(i,i)=1$ for all $1\leq i \leq n$)}.
$$
Let $k^{(n)}$ be a sequence in $\mathcal{K}_\lambda$ which converges to $k$ in matrix norm. Then by continuity, it follows that $k$ satisfies all the conditions in (\ref{admiss conds}). Note that for $k$ to be a member of $\mathcal{K}_\lambda$, $k$ has to be strictly positive definite as a matrix. Suppose on the contrary that there exists a vector $0\neq v=(v_1,v_2,\dots, v_n)\in \mathbb{C}^n$ such that $kv=0$. Let $\Lambda_s$ and $\Lambda_p$ denote the diagonal matrix whose $(i,i)$th entries are $s_i$ and $p_i$ respectively. Then we have
\begin{eqnarray*}
0&\leq& \sum_{i,j}(4-s_i\bar{s_j})k(i,j)\bar{v_i}v_j
 \\
 &=& 4\sum_{i,j}k(i,j)\bar{v_i}v_j -\sum_{i,j}k(i,j)s_j v_j\overline{s_i v_i}
 \\
 &=& 4\langle k v,v \rangle -\langle k \Lambda_sv, \Lambda_sv \rangle,
\end{eqnarray*}
which implies that $k(\Lambda_sv)=0$. Similarly, using the second condition in (\ref{admiss conds}) we get $k(\Lambda_pv)=0$. Continuing this way, we get by induction that if $m$ and $n$ are two natural numbers, then $k(\Lambda_s^m\Lambda_p^nv)=0$, which in turn implies that
\begin{eqnarray}\label{contradiction}
k(f(\Lambda_s,\Lambda_p)v)=0 \text{ for every polynomial $f$ in two variables.}
\end{eqnarray}
Note that $f(\Lambda_s,\Lambda_p)$ is the diagonal operator whose $i$th diagonal entry is $f(s_i,p_i)$. Since $v$ is a non-zero vector, it has to have a non-zero co-ordinate. Let it be $v_j$ for some $j$ between $1$ and $n$. Since $(s_1,p_1),(s_2,p_2),\dots,(s_n,p_n)$ are assumed to be distinct, there exists a polynomial $f$ such that $f(s_i,p_i)=0$ for every $i\neq j$ and $f(s_j,p_j)=1$. For this $f$, the vector $f(\Lambda_s,\Lambda_p)v$ has $v_j$ at the $j$th entry and zero elsewhere. When the matrix $k$ is applied to it, one gets the $j$th column of $k$ multiplied by $v_j$. This vector is non-zero because its $j$th entry is $v_j$ which is non-zero. This contradicts (\ref{contradiction}). Therefore $k$ must be strictly positive and hence $\mathcal{K}_\lambda$ is compact.

Let $f$ be an extremal for (\ref{extremal}). Then by Theorem \ref{Interpolation-Reformulation} we have
\begin{eqnarray}\label{extremalf}
\left( \left( \; (\rho^2 - w_i \bar{w}_j) k(i,j) \; \right) \right) \geq 0
\end{eqnarray}
for all $k \in \mathcal K_\lambda$. Since $\mathcal{K}_\lambda$ is compact, there exists a kernel $\delta$ in $\mathcal{K}_\lambda$ and a non-zero vector $x=(x_1,x_2,\dots,x_n)$ such that \
\begin{eqnarray}\label{zero}
\sum_{i,j}(\rho^2 - w_i \bar{w}_j)\delta(i,j)\bar{x_i}x_j = 0.
\end{eqnarray}
Since $\delta$ is strictly positive definite, $\{\delta(\cdot, j): 1 \leq j \leq n\}$ forms a linearly independent set of vectors. Hence the operators $S^*$ and $P^*$ defined by
$$
S^*\delta(\cdot , j) =\bar{s_j}\delta(\cdot ,j) \text{ and  } P^*\delta ( \cdot ,j) =\bar{p_j}\delta (\cdot ,j) \text{ for all $1\leq j \leq   n$}
$$
are uniquely defined on the $n$-dimensional vector space spanned by $\{\delta(\cdot, j): 1 \leq j \leq n\}$. The way $S^*$ and $P^*$ are defined, for each $j=1,2, \ldots ,n$, the pair of numbers $(\bar{s_j},\bar{p_j})$ is a joint eigenvalue for $(S^*, P^*)$ with the corresponding eigenvectors $\delta(\cdot , j)$. Since $(S^*,P^*)$ is a pair of commuting operators on a finite dimensional space, its Taylor joint spectrum is the same as the set of joint eigenvalues which is just the set $\{(\bar{s_j},\bar{p_j}):1\leq j \leq n\}$, which implies that $(S,P)$ is subordinate to $\{(s_j,p_j):1\leq j \leq n\}$. Since $\delta \in \mathcal{K}_\lambda$, the pair $(S,P)$ is a $\Gamma$-contraction by the characterization we obtained in Lemma \ref{admi}. Also note that if $f$ is an extremal for (\ref{extremal}), then for the function $\check{f}$ defined by $\check{f}(s,p)=\overline{f(\bar{s}, \bar{p})}$, we have
$$
f(S,P)^*\delta(\cdot , j) =\check{f}(S^*,P^*)\delta (\cdot , j) =\overline{f(s_j,p_j)}\delta (\cdot , j) =\bar{w_j}\delta(\cdot , j).
$$ By (\ref{extremalf}), it follows that $\|f(S,P)\|\leq \rho$. Now, (\ref{zero}) implies that $\|f(S,P)\|=\rho.$ This proves the lemma.
\end{proof}

The proof of the extension theorem now follows easily. Let $V$ be an $\mathcal A$-von Neumenn set and let $f$ be in $\mathcal A$. Choose a dense set $\{ \lambda_1, \lambda_2, \ldots \}$ for $V$. Let $\rho_n$ be the same as $\rho$ above, now the dependence on $n$ being shown explicitly. Let $f_n$ be an extremal for $\rho_n$. Let $(S_n,P_n)$ be a $\Gamma$-contraction subordinate to $\{ \lambda_1, \lambda_2, \ldots ,\lambda_n \}$ as obtained in the lemma above. Then
\begin{eqnarray*}
\| f_n \| & = & \rho_n \mbox{ because } f_n \mbox{ is extremal for } \rho_n \\
& = & \| f_n(S_n , P_n) \|  \mbox{ by the Lemma above} \\
& = & \| f(S_n , P_n) \| \mbox{ because } (S_n , P_n) \mbox{ is subordinate to } \{ \lambda_1, \lambda_2, \ldots ,\lambda_n \} \\
& & \mbox{ and } f_n = f \mbox{ on } \{ \lambda_1, \lambda_2, \ldots ,\lambda_n \}\\
& \le & \| f \|_V \mbox{ because } (S_n , P_n) \mbox{ is a } \Gamma-\mbox{contraction subordinate to } V \\
& & \mbox{ and } V \mbox{ is an $\mathcal A$-von Neumann set.}
\end{eqnarray*}

Now, by a compactness argument again, this time in the weak* topology of $H^\infty$, there exists a $g$ in $H^\infty (\mathbb G)$ such that $f_n \rightarrow g$ pointwise (and hence $g = f$ on $V$) and $\| g \| \le \| f \|$ because the $f_n$ are dominated by $f$ in norm. Thus, $g$ is an extension of $f$.

The Extension Theorem has also been proved in \cite{ALYrecent} by a different method.

\section{Epilogue}

\begin{enumerate}

\item There is a dual way of expressing the Interpolation Theorem.

\begin{thm}
Let $\lambda_1, \lambda_2, \ldots ,\lambda_n \in \mathbb G$ and $w_1, w_2, \ldots , w_n \in \Dbar$. There exists a function $f$ in the closed unit ball of $H^\infty(\mathbb G)$ that interpolates each $\lambda_i$ to $w_i$ if and only if there is a $C(\Dbar)^*$ valued positive semidefinite kernel $\Delta$ on $\{ \lambda_1, \lambda_2, \ldots ,\lambda_n \} \times \{ \lambda_1, \lambda_2, \ldots ,\lambda_n \}$ such that
$$ 1 - w_i \overline{w_j} = \Delta (\lambda_i , \lambda_j ) \left( 1 - \varphi (\cdot , \lambda_i) \overline{\varphi (\cdot , \lambda_j)} \right).$$
\end{thm}

\begin{proof}
Necessity is obvious by part (\textbf{D}) of the Realization Theorem.

For sufficiency, we apply Lemma \ref{GNS} to get a Hilbert space $H$, a function $L : \mathbb G \rightarrow B( C(\Dbar) , H)$ and a unital $*$-representation $\pi$ of $C(\Dbar)$ on $H$ so that
$$ 1 + \Delta (\lambda_i , \lambda_j) \big( \varphi (\cdot , \lambda_i) \overline{\varphi (\cdot , \lambda_j)} \big) = w_i \overline{w_j} + \Delta (\lambda_i , \lambda_j) 1.$$
Hence
 $$ 1 +   \langle \pi \varphi(\cdot , \lambda_i) L (\lambda_i) 1,  \pi \varphi(\cdot , \lambda_j) L (\lambda_j) 1 \rangle = w_i \overline{w_j} + \langle L (\lambda_i) 1 , L (\lambda_j) 1 \rangle.$$
From here on, the rest of the proof follows the proof of the Interpolation Theorem in Section 4.
\end{proof}

\item Consider the situation of (\ref{ExampleofDelta}). Let $H_\alpha$ be the Hilbert space corresponding to $\delta(\alpha, \cdot, \cdot)$. Then the ingredients of the Realization Theorem can be taken to be direct integrals:
$$H = \int_{\Dbar}^\oplus H_\alpha d\mu(\alpha) \mbox{ and } \pi(f) =  \int_{\Dbar}^\oplus f(\alpha) d\mu(\alpha)$$ for $f \in C(\Dbar)$. Thus, in this case,
$$\pi(\varphi(\cdot , s, p)) = \int_{\Dbar}^\oplus \varphi(\cdot , s, p) d\mu .$$
Compare this with $\mathcal E_\lambda$ of Theorem 11.13 of \cite{AMc-book}. There were only two co-ordinates there - $z_1$ and $z_2$ and hence two Hilbert spaces, so $\mathcal E_\lambda$ acted on the direct sum of $H_1$ and $H_2$. In the case of the symmetrized bidisk, there are uncountably many parametrized co-ordinate functions and hence a direct integral is needed.

\end{enumerate}

\end{document}